\documentclass[12pt]{amsart}
\usepackage{amsmath,amssymb,amsbsy,amsfonts,amsthm,latexsym,amsopn,amstext,
                                                    amsxtra,euscript,amscd}
\usepackage[english]{babel}
\begin{document}

\newtheorem{thm}{Theorem}
\newtheorem{lem}[thm]{Lemma}
\newtheorem{claim}[thm]{Claim}
\newtheorem{cor}[thm]{Corollary}
\newtheorem{prop}[thm]{Proposition} 
\newtheorem{definition}{Definition}
\newtheorem{question}[thm]{Open Question}
\newtheorem{conj}[thm]{Conjecture}
\newtheorem{prob}{Problem}
\def\vol {{\mathrm{vol\,}}}
\def\squareforqed{\hbox{\rlap{$\sqcap$}$\sqcup$}}
\def\qed{\ifmmode\squareforqed\else{\unskip\nobreak\hfil
\penalty50\hskip1em\null\nobreak\hfil\squareforqed
\parfillskip=0pt\finalhyphendemerits=0\endgraf}\fi}

\def\cA{{\mathcal A}}
\def\cB{{\mathcal B}}
\def\cC{{\mathcal C}}
\def\cD{{\mathcal D}}
\def\cE{{\mathcal E}}
\def\cF{{\mathcal F}}
\def\cG{{\mathcal G}}
\def\cH{{\mathcal H}}
\def\cI{{\mathcal I}}
\def\cJ{{\mathcal J}}
\def\cK{{\mathcal K}}
\def\cL{{\mathcal L}}
\def\cM{{\mathcal M}}
\def\cN{{\mathcal N}}
\def\cO{{\mathcal O}}
\def\cP{{\mathcal P}}
\def\cQ{{\mathcal Q}}
\def\cR{{\mathcal R}}
\def\cS{{\mathcal S}}
\def\cT{{\mathcal T}}
\def\cU{{\mathcal U}}
\def\cV{{\mathcal V}}
\def\cW{{\mathcal W}}
\def\cX{{\mathcal X}}
\def\cY{{\mathcal Y}}
\def\cZ{{\mathcal Z}}

\def\NmQR{N(m;Q,R)}
\def\VmQR{\cV(m;Q,R)}

\def\Xm{\cX_m}

\def \C {{\mathbb C}}
\def \F {{\mathbb F}}
\def \L {{\mathbb L}}
\def \K {{\mathbb K}}
\def \Q {{\mathbb Q}}
\def \R {{\mathbb R}}
\def \Z {{\mathbb Z}}
\def \fS{\mathfrak S}

\def\\{\cr}
\def\({\left(}
\def\){\right)}
\def\fl#1{\left\lfloor#1\right\rfloor}
\def\rf#1{\left\lceil#1\right\rceil}

\def\Tr{{\mathrm{Tr}}}
\def\Im{{\mathrm{Im}}}

\def \bFp {\overline \F_p}

\newcommand{\pfrac}[2]{{\left(\frac{#1}{#2}\right)}}

\def \Prob{{\mathrm {}}}
\def\e{\mathbf{e}}
\def\ep{{\mathbf{\,e}}_p}
\def\epp{{\mathbf{\,e}}_{p^2}}
\def\em{{\mathbf{\,e}}_m}

\def \bS {\mathfrak{S}}

\def\Res{\mathrm{Res}}

\def\vec#1{\mathbf{#1}}
\def\flp#1{{\left\langle#1\right\rangle}_p}

\def\mand{\qquad\mbox{and}\qquad}

\newcommand{\comm}[1]{\marginpar{%
\vskip-\baselineskip 
\raggedright\footnotesize
\itshape\hrule\smallskip#1\par\smallskip\hrule}}

\title{Double Character Sums over Subgroups and Intervals}

\author[M.-C. Chang]
{Mei-Chu Chang}
\address{Department of Mathematics, University of California,
Riverside,  CA 92521, USA}
\email{mcc@math.ucr.edu}

\author{Igor E. Shparlinski} 
\address{Department of Pure Mathematics, University of New South Wales, 
Sydney, NSW 2052, Australia}
\email{igor.shparlinski@unsw.edu.au}

\date{\today}

\begin{abstract} We estimate double sums
$$
S_\chi(a, \cI, \cG) =
\sum_{x \in \cI} \sum_{\lambda \in \cG} 
\chi(x + a\lambda), \qquad  1\le a < p-1,
$$
with a multiplicative character $\chi$ modulo $p$
where $\cI= \{1,\ldots, H\}$ and $\cG$ is a 
subgroup of order $T$ of the multiplicative 
group of the finite field of $p$ elements. 
A nontrivial upper bound on $S_\chi(a, \cI, \cG) $
can be derived from the Burgess bound if 
$H \ge p^{1/4+\varepsilon}$ and from
some standard elementary arguments if 
$T \ge p^{1/2+\varepsilon}$, where $\varepsilon>0$
is arbitrary. We obtain a nontrivial estimate 
in a wider range of parameters $H$ and $T$. 
We also estimate double sums
$$
T_\chi(a, \cG) =
\sum_{\lambda, \mu \in \cG} 
\chi(a + \lambda + \mu), \qquad  1\le a < p-1,
$$
and give an application to primitive roots modulo $p$ with 
$3$ non-zero binary digits.
\end{abstract}



\maketitle

\section{Introduction}

\subsection{Background and motivation}

For a prime $p$, we use $\F_p$ to denote the finite field  of $p$ elements,
which we always assume to be represented by the set $\{0, \ldots, p-1\}$.

Since the spectacular results of 
Bourgain,  Glibichuk \& Konyagin~\cite{BGK}, 
Heath-Brown \& Konyagin~\cite{HBK} and 
Konyagin~\cite{Kon}
on bounds of exponential sums 
\begin{equation}
\label{eq:ExpSum}
\sum_{\lambda \in \cG} \exp(2 \pi i a \lambda /p), \qquad 
a \in \F_p^*,
\end{equation}
over small
multiplicative subgroups $\cG$ of $\F_p^*$, there has been 
a remarkable progress in this direction, also involving 
sums over   consecutive powers $g^i$,  $i=1, \ldots, N$,
of elements $g \in \F_p^*$, see the survey~\cite{Gar}
and also very recent results of Bourgain~\cite{Bour3,Bour4}
and Shkredov~\cite{Shkr1,Shkr2}. Exponential sums 
over short segments of 
consecutive powers $g, \ldots, g^N$ of a fixed element $g \in \F_p^*$,
have also been studied, see~\cite{Kerr,KonShp} and references therein. 
However the multiplicative analogues of the sums~\eqref{eq:ExpSum}, 
that is, the sums
$$
\sum_{\lambda \in \cG} \chi(a +\lambda), \qquad 
a \in \F_p^*,
$$
with a nonprincipal multiplicative character $\chi$ of $\F_p$ 
have been resisting all attempts to improve the classical 
bound 
\begin{equation}
\label{eq:CharSum}
\left|\sum_{\lambda \in \cG} \chi(a +\lambda) \right| \le 
\sqrt{p}.
\end{equation}
Note that~\eqref{eq:CharSum} is instant from the Weil bound, 
see~\cite[Theorem~11.23]{IwKow}, if one notices
that 
$$
\sum_{\lambda \in \cG} \chi(a +\lambda) =
\frac{T}{p-1} \sum_{\mu \in \F_p^*} \chi(a +\mu^{(p-1)/T}), 
$$
where $T = \#\cG$ (but can also be obtained 
via elementary arguments). 

We now recall that Bourgain~\cite[Section~4]{Bour2} has shown that double sums
over short intervals and short segments of consecutive 
powers 
$$
\sum_{x =1}^H \sum_{n=1}^N 
\exp(2 \pi i a x g^n/p), \qquad  1\le a < p-1,
$$
can be estimated for much smaller values of $N$ than for single sums over 
consecutive powers. Here we show that similar mixing 
can also be applied to the sums of multiplicative characters 
and thus lead to nontrivial estimates of the sums 
$$
S_\chi(a, \cI, \cG) =
\sum_{x \in \cI} \sum_{\lambda \in \cG} 
\chi(x + a\lambda), \qquad  1\le a < p-1,
$$
where $\cI= \{1,\ldots, H\}$ is an interval of $H$ 
consecutive integers and $\cG\subseteq \F_p^*$ is a 
multiplicative subgroup of order $T$ for the values 
of $H$ and $T$ to which previous bounds do not apply.
More precisely, one can immediately estimate the sums
 $S_\chi(a, \cI, \cG) $ nontrivially if for some
 fixed  $\varepsilon>0$ we have  
$H \ge p^{1/4+\varepsilon}$, by using the Burgess bound,
 see~\cite[Theorem~12.6]{IwKow}, or  
$T \ge p^{1/2+\varepsilon}$, by using~\eqref{eq:CharSum}.

\subsection{Main results}

Here we obtain a nontrivial estimate 
in a wider range of parameters $H$ and $T$. 

\begin{thm} \label{thm:Bound IG1} For every fixed real $\varepsilon>0$ 
there are some $\delta> 0$ and $\eta> 0$ such that 
if $H> p^\varepsilon$ and $T > p^{1/2 - \delta}$ then 
for the interval $\cI= \{1,\ldots, H\}$ and the multiplicative subgroup
$\cG \subseteq \F_p^*$ of order $T$, we have
$$
S_\chi(a, \cI, \cG) = O(HT p^{-\eta})
$$
uniformly over $a \in \F_p^*$ and
nonprincipal  multiplicative characters $\chi$ of $\F_p$. 
\end{thm}

We also obtain a similar estimate at the other end of 
region of $H$ and $T$, namely for a very small $T$ and 
$H$ that is still below the reach of the Burgess bound
(see~\cite[Theorem~12.6]{IwKow}). In fact in this case we
are able to estimate a more general sums
$$
\bS_\chi(f, \cI, \cG) =
\sum_{x \in \cI} \sum_{\lambda \in \cG} 
\chi(x + f(\lambda)),  
$$
with a non-constant polynomial $f\in \F_p[X]$. 

\begin{thm} \label{thm:Bound IG2} For every fixed real $\varepsilon>0$ and
integer $d\ge 1$ there are some $\delta> 0$ and $\eta> 0$ such that 
if $T> p^\varepsilon$ and $H > p^{1/4 - \delta}$ then 
for the interval $\cI= \{1,\ldots, H\}$, the multiplicative subgroup
$\cG \subseteq \F_p^*$ of order $T$, we have
$$
\bS_\chi(f, \cI, \cG) = O(HT p^{-\eta})
$$
uniformly over polynomials $f\in \F_p[X]$ of degree $d$ and
nonprincipal  multiplicative characters $\chi$ of $\F_p$. 
\end{thm}

We also give an explicit version of Theorem~\ref{thm:Bound IG1}
in the case when $H = p^{1/4+o(1)}$ and 
$T = p^{1/2+o(1)}$, that is, when other methods
just start to fail. 

\begin{thm} \label{thm:Bound IG Expl} Let $H = p^{1/4+o(1)}$ and 
$T = p^{1/2+o(1)}$. Then 
for the interval $\cI= \{1,\ldots, H\}$ and the multiplicative subgroup
$\cG \subseteq \F_p^*$ of order $T$, we have
$$
|S_\chi(a, \cI, \cG) | \le HT p^{-5/48 + o(1)}
$$
uniformly over $a \in \F_p^*$ and
nonprincipal  multiplicative characters $\chi$ of $\F_p$. 
\end{thm}

Furthermore, we also consider double sums 
$$
T_\chi(a, \cG) =
 \sum_{\lambda,\mu \in \cG} 
\chi(a + \lambda+\mu), \qquad  1\le a < p-1,
$$
where both variables
run over a  multiplicative subgroup
$\cG \subseteq \F_p^*$.

Using recent estimates of Shkredov~\cite{Shkr1} on the so-called
{\it additive energy\/} of multiplicative 
subgroups we also estimate them below the obvious range $T \ge p^{1/2}$, 
where $T = \# \cG$, given by the estimate 
$$
|T_\chi(a,  \cG) | \le T p^{1/2}, 
$$
which follows from~\eqref{eq:CharSum}.

\begin{thm} \label{thm:Bound GG} Let  
$T \le p^{2/3}$. Then 
for  the multiplicative subgroup
$\cG \subseteq \F_p^*$ of order $T$, we have
$$
|T_\chi(a, \cG) | \le 
\left\{ \begin{array}{ll}
T^{19/26} p^{1/2 +o(1)},&  \textrm{if $T \le p^{1/2}$},\\
T^{9/13} p^{27/52+o(1)},&  \textrm{if $p^{1/2} < T \le p^{29/48}$},\\
T p^{1/3+o(1)},&  \textrm{if $p^{29/48} < T \le p^{2/3}$},\\
\end{array} \right.
$$
uniformly over $a \in \F_p^*$ and
nonprincipal  multiplicative characters $\chi$ of $\F_p$. 
\end{thm}

Note that Theorem~\ref{thm:Bound GG}  nontrivial 
provided that $T \ge p^{13/33+\varepsilon}$ for some fixed $\varepsilon > 0$.

We also give an application of Theorem~\ref{thm:Bound GG} to primitive 
roots modulo $p$ with few non-zero binary digits. 
More precisely, let $u_p$ denote the smallest $u$ such that 
there exists a primitive root modulo $p$ with $u_p$ non-zero binary digits.
It is shows in~\cite[Theorem~5]{DES} that $u_p\le 2$ for all but $o(Q/\log Q)$
primes $p\le Q$, as $Q\to \infty$ (note that in~\cite{DES} the result 
is formulated only for quadratic non-residues but it is easy to see
that the argument also holds for primitive roots). Instead of $o(Q/\log Q)$, 
can obtain a slightly
more explicit but still rather weak bound on the size of the exceptional set.
Here we show that Theorem~\ref{thm:Bound GG} implies a rather 
strong bound on the set of primes $p\le Q$ for which $u_p \le 3$
does not hold.

\begin{thm} \label{thm:up} 
For all but at most $Q^{26/33+o(1)}$ primes $p \le Q$, we have 
$u_p \le 3$ as $Q\to \infty$. 
\end{thm}

We also note that one may attempt to treat the sums $S_\chi(a, \cI, \cG)$ and 
$T_\chi(a, \cG)$ 
within the general theory of double sums of multiplicative 
characters, see~\cite{BGKS1,BGKS2,BKS3,Chang,FrIw,Kar1,Kar2,Kar3}
and references therein. However it seems that none of the presently 
known results implies a nontrivial estimate in the range
of Theorems~\ref{thm:Bound IG1} and~\ref{thm:Bound GG}.

\section{Preparations}

\subsection{Notation and general conventions}

Throughout the paper, $p$ always denotes a sufficiently 
large prime number and $\chi$ denotes an non-principal 
multiplicative character modulo $p$. We assume that $\F_p$ is
represented by the set $\{0, \ldots, p-1\}$. 

Furthermore, $\cG$ always denotes a multiplicative subgroup of $\F_p^*$
of order $\# \cG = T$ and $\cI$ always denotes the set
$\cI= \{1,\ldots, H\}$.  

We also assume that $f \in \F_p[X]$ is a of degree $d \ge 1$.
In particular, $f$ is not a constant. 

The notations $U = O(V)$ and $U \ll V$   are both
equivalent to the inequality $|U| \le c\,V$ 
with some constant $c> 0$ that may depend on 
the real parameter $\varepsilon > 0$ and the integer parameters $d\ge 1$
and $\nu \ge 1$
and is absolute otherwise.

In particular, all our estimates are uniform with respect to the polynomial
$f$ and the character $\chi$.

\subsection{Bounds of some exponential and character sums}

First we recall the classical result of Davenport and 
Erd{\H o}s~\cite{DavErd}, which 
 follows from the Weil bound of multiplicative character sums,
see~\cite[Theorem~11.23]{IwKow}.

\begin{lem}
\label{lem:DavErd1} 
For a fixed integer $\nu \ge 1$ and an 
integer $R < p$, we have
$$
 \sum_{v \in \F_p}  
 \left|\sum_{r =1}^R
\chi\(v + r\)\right|^{2\nu} \ll R^{2\nu} p^{1/2} + R^\nu p.
$$
\end{lem}

The following result is a version of Lemma~\ref{lem:DavErd1} 
with $\nu=1$ which is slightly more precise in this case. 

\begin{lem}
\label{lem:DavErd2} 
 For any set $\cV \subseteq \F_p$ and 
 complex numbers $\alpha_v$ of
such that  $|\alpha_v| \le 1$ for $v \in \cV$, 
we have 
$$
\sum_{u \in \F_p} \left|\sum_{v \in \cV} \chi(u+v) \right|^{2} 
\ll \#\cV p.
$$
\end{lem}

\begin{proof} Denoting by $\overline \chi$ the conjugate character and 
recalling that $\overline \chi(w) = \chi(w^{-1})$ for $w \in \F_p^*$, we
obtain
$$
\sum_{u \in \F_p} \left|\sum_{v \in \cV} \chi(u+v) \right|^{2} 
= \sum_{v,w \in \cV}  \alpha_v \overline  \alpha_w  \sum_{u \in \F_p} \chi(u+v) \overline \chi(u+w).
$$
If $v =w$ the inner sum is equal to $p-1$. So the total contribution from 
such terms is $O(Mp)$. Otherwise, we derive
\begin{equation*}
\begin{split}
 \sum_{u \in \F_p} \chi(u+v) &\overline \chi(u+w) = 
  \sum_{u \in \F_p} \chi(u+v -w) \overline \chi(u)\\
&= \sum_{u \in \F_p^*} \chi(u+v -w) \overline \chi(u)
= \sum_{u \in \F_p^*} \chi\(1+(v -w)u^{-1}\) \\
&= \sum_{u \in \F_p^*} \chi\(1+u\) = \sum_{u \in \F_p} \chi\(1+u\)  - \chi(1)
= - \chi(1).
\end{split}
\end{equation*}
 So the total contribution from 
such terms is $O(M^2) = O(Mp)$
and the result follows.
\end{proof}

We also need the following bound of Bourgain~\cite[Theorem~1]{Bour1}.

\begin{lem}
\label{lem:SparseExpSum} 
For every fixed real $\varepsilon>0$ and
integer $r\ge 1$ there is some  $\xi>  0$ such that 
for any integers $k_1, \ldots, k_r \ge 1$ with 
$$\gcd(k_i,  p-1) < p^{1-\varepsilon}, \mand
\gcd(k_i-k_j, p-1) < p^{1-\varepsilon}, 
$$ 
for $i, j =1, \ldots, r$,  $i \ne j$, uniformly over the coefficients $a_1, \ldots, a_r\in \F_p$, not all equal to zero, we have
$$
\sum_{x=1}^{p-1} \exp\(\frac{2 \pi i}{p} \(a_1x^{k_1} + \ldots + a_r x^{k_r}\)\) 
\ll p^{1-\xi}.
$$
\end{lem}

Clearly for any $F \in \F_p[X]$ and a multiplicative subgroup   $\cG \subseteq \F_p^*$
of order $\#\cG = T$ we have 
$$
\frac{1}{\#\cG} \sum_{\lambda \in \cG} 
\exp\(\frac{2 \pi i}{p} F(\lambda)\) =\frac{1}{p-1} 
\sum_{x=1}^{p-1} \exp\(\frac{2 \pi i}{p} F(x^{(p-1)/T})\)
\ll p^{-\xi}.
$$
so we derive from Lemma~\ref{lem:SparseExpSum}:

\begin{cor}
\label{cor:SubgrExpSum} 
For every fixed real $\varepsilon>0$ and
integer $d\ge 1$ there is some  $\xi > 0$ such that 
for  $T \ge p^\varepsilon$, uniformly over $a \in \F_p^*$, we have
$$
 \sum_{\lambda \in \cG} 
\exp\(\frac{2 \pi i}{p} a f(\lambda)\) \ll T p^{-\xi}.
$$
\end{cor}

\subsection{Bound on the number of solutions to some
congruences}

First we note that combining 
Corollary~\ref{cor:SubgrExpSum} with the  {\it Erd\H{o}s-Tur\'{a}n inequality}
(see, for example,~\cite[Theorem~1.21]{DrTi}) that relates the uniformity of distribution to exponential sums, we immediately obtain:

\begin{lem}
\label{lem:G in I}  For every fixed real $\varepsilon>0$ and
integer $r\ge 1$ there is some  $\kappa > 0$ such that  
for   $T \ge p^\varepsilon$, we have
$$
\# \{\lambda \in \cG~:~ f(\lambda) \equiv b+x \pmod p, \text{ where } x \in \cI\}
= \frac{H T}{p} + O\(T^{1-\kappa}\),
$$
uniformly over $b \in \F_p$.
\end{lem}

Let  $N(\cI,\cG)$ be the number of solutions to the congruence
$$
\lambda x\equiv y \pmod p, \qquad  x,y \in \cI, \  \lambda\in \cG.
$$

Some of our results rely on an upper bound on $N(\cI,\cG)$ which is given 
in~\cite[Theorem~1]{BKS1}, see also~\cite{BKS2} for some other bounds. 

\begin{lem}
\label{lem:NIG} Let $\nu\ge 1$ be a fixed
integer.  Then 
$$
N(\cI,\cG) \le H t^{(2\nu +1)/2\nu(\nu+1)}p^{-1/2(\nu +1) + o(1)}
+ H^2 t^{1/\nu}p^{-1/\nu + o(1)}, 
$$
as $p\to \infty$, where 
$$
t = \max\{T, p^{1/2}\}.
$$
\end{lem}

We also use the following bound which 
is due to Ayyad, Cochrane and Zheng~\cite[Theorem~1]{ACZ}.

\begin{lem}
   \label{lem:SymCong} 
Let $\cJ_i = \{b_i+1, \ldots, b_i + h_i\}$
for some integers $p> h_i + b_i > b_i \ge 1$, $i =1, 2, 3, 4$. Then 
\begin{equation*}
\begin{split}
\# \{(x_1,x_2,x_3,x_4) \in \cJ_1&\times \cJ_2 \times \cJ_3 \times \cJ_4~:~
x_1 x_2 \equiv x_3 x_4\pmod p\} \\
& \quad = \frac{1}{p} h_1 h_2 h_3 h_4 + O\(\( h_1 h_2 h_3 h_4\)^{1/2}
(\log p)^2\).
\end{split}
\end{equation*}
\end{lem}

We now fix some real $L > 1$ and
denote by $\cL$  the set of primes of the interval $[L, 2L]$.
We need an upper bound on the quantity
\begin{equation}
\label{eq:def W}
\begin{split}
W  =\# \Bigl\{(u_1,u_2,\ell_1,\ell_2,s_1,s_2)&\in \cI^2\times \cL^2 \times \cS^2~:\\
&~
\frac{u_1+s_1}{\ell_1} \equiv \frac{u_2+s_2}{\ell_2} \pmod p\Bigr\}
\end{split}
\end{equation}
for some special class of sets. 

We say that a set $\cS \subseteq \F_p$ is $h$-spaced
if no elements $s_1, s_2\in \cS$ and positive integer $k\le h$
satisfy the equality $s_1+k=s_2$.

The following result is given in~\cite{BKS3} 
and is based on some ideas of Shao~\cite{Shao}.

\begin{lem}
\label{lem:W}  If $L < H$ and $2HL < p$ then for any 
$H$-spaced set $\cS$ for $W$, given by~\eqref{eq:def W}
we have
$$
W\ll \frac{(\# \cS H L)^2}{p}  +  \# \cS H L p^{o(1)}.
$$
\end{lem}

We also define
\begin{equation}
\label{eq:def U}
U=\sum_{v\in \F_p} U(v)^2,
\end{equation}
where 
\begin{equation}
\label{eq:def Uv}
U(v) =\# \left\{(u,\ell,\lambda)\in \cI \times \cL  \times \cG~:~
\frac{u +f(\lambda)}{\ell} \equiv v \pmod p\right\}.
\end{equation}

\begin{lem}
\label{lem:Uv}  For every fixed real $\varepsilon>0$ and
integer $d\ge 1$ there  are some $\delta>$ and $\eta> 0$ such that 
if 
$$T> p^\varepsilon \mand  p^{1/2-\varepsilon} \ge H \ge L
$$
then for $U$, given by~\eqref{eq:def U}
we have
$$
U\ll H   L  T^2  p^{-\eta}.
$$
\end{lem}

\begin{proof}  Let $\cS_1$ be the largest $H$-separated subset 
of $\cF_0 = \{f(\lambda)~:~\lambda \in \cG\}$. 
By Lemma~\ref{lem:G in I} we have $\# \cS_1 \gg p^\kappa$
for some fixed $\kappa>0$. 

Inductively, we define  $\cS_{k+1}$ as the largest 
$H$-separated subset 
of 
$$
\cF_{k} = \cF_{k-1} \setminus \bigcup_{j=1}^k \cS_j, 
\qquad k =1, 2, \ldots.
$$
Clearly for some $b \in \F_p$ and a set 
$\cJ = \{b+1, \ldots, b+H\}$ we have 
$$
\# \(\cF_{k} \cap \cJ\)\ge \frac {\# \cF_{k}} {\# \cF_{k+1} }.
$$
On the other hand, by Lemma~\ref{lem:G in I}
$$
\# \(\cF_{k} \cap \cJ\) \le \# \(\cF_{1} \cap \cJ\) \ll Tp^{-\kappa}.
$$
Hence there is a partition 
$$
\cF_0 = \bigcup_{k=0}^{K} \cS_k 
$$
into disjoined sets with $K \le T p^{-\kappa/2}$ such that
\begin{itemize}
\item $\# \cS_{0} \le T p^{-\kappa/2}$,
\item  $\cS_k$ is $H$-separated  with  $\# \cS_k \ge p^{\kappa/2}$, 
$k=1, \ldots, K$.
\end{itemize}
For $k=0, \ldots, K$ we define
$$
U_k(v) =\# \left\{(u,\ell,s)\in \cI \times \cL  \times \cS_k~:~
\frac{u +s}{\ell} \equiv v \pmod p\right\}.
$$
We have 
$$
U(v) = \sum_{k=0}^K U_k(v)= U_0(v) +  \sum_{k=1}^K U_k(v).
$$
So, squaring out and summing over all $v\in \F_p$, we obtain
\begin{equation*}
\begin{split}
U & \ll    
\sum_{v\in \F_p} U_0(v)^2 +
\sum_{v\in \F_p} \(\sum_{k=1}^K  U_k(v)\)^2\\
& = \sum_{v\in \F_p} U_0(v)^2 +
\sum_{v\in \F_p} \sum_{k,m=1}^K   U_k(v) U_m(v).
\end{split}
 \end{equation*}
 Now,  changing 
the order of summation in the second term in the above  and  then using the Cauchy 
inequality, yields
\begin{equation}
\label{eq:UUkU0}
U  \ll
 V_1 + V_2^2,
 \end{equation}
 where
$$
V_1 = \sum_{v\in \F_p} U_0(v)^2 \mand 
V_2 = \sum_{k=1}^K \(\sum_{v\in \F_p} U_k(v)^2\)^{1/2}.
$$
 We have, 
\begin{equation*}
\begin{split}
V_1 
= \# \Bigl\{(u_1,u_2,\ell_1,\ell_2,s_1,s_2)&\in \cI^2 \times \cL^2  \times \cS_0^2~:\\
&~\frac{u_1 +s_1}{\ell_1} \equiv \frac{u_2 +s_2}{\ell_2}  \pmod p\Bigr\}\\
\le \max_{s_1,s_2 \in \F_p}  \# \Bigl\{(u_1,u_2,\ell_1,\ell_2)&\in \cI^2 \times \cL^2  ~:\\
&~\frac{u_1 +s_1}{\ell_1} \equiv \frac{u_2 +s_2}{\ell_2}  \pmod p\Bigr\}.
\end{split}
\end{equation*}
Since $L \le H \le p^{1/2-\varepsilon}$, by Lemma~\ref{lem:SymCong}
we obtain 
\begin{equation}
\label{eq:V1}
V_1  \ll    
(\# \cS_0)^2 H L (\log p)^2 \ll HLT^2  p^{-\varepsilon}(\log p)^2.
 \end{equation}
 
Furthermore, Lemma~\ref{lem:W} implies that for $k =1, \ldots, K$
we have 
$$
 \sum_{v\in \F_p}   U_k(v)^2 
\ll (\# \cS_k H L)^2p^{-1}  +  \# \cS_k H L p^{o(1)}.
$$
Hence, applying the Cauchy inequality, we derive
\begin{equation*}
\begin{split}
V_2  & \ll  \sum_{k=1}^K  \(\# \cS_k H Lp^{-1/2}  +  (\# \cS_k)^{1/2} H^{1/2} L^{1/2} p^{o(1)}\) \\ 
& \le H L T p^{-1/2}  +  H^{1/2} L^{1/2} p^{o(1)}\sum_{k=1}^K (\# \cS_k)^{1/2} \\
& \le H L T p^{-1/2}  +  H^{1/2} L^{1/2} p^{o(1)} \(K\sum_{k=1}^K \# \cS_k\)^{1/2}\\
& \le H L T p^{-1/2}  +  H^{1/2} K^{1/2} L^{1/2} T^{1/2} p^{o(1)}.  
\end{split}
 \end{equation*}
 Since  $K \le T p^{-\kappa/2}$  and $L \le H \le p^{1/3}$, we see that
\begin{equation}
\label{eq:V2}
V_2  \ll H LT p^{-1/2}  +   H^{1/2}  L^{1/2} T  
p^{ -\kappa/2 +o(1)} \le H^{1/2}  L^{1/2} T  p^{ -\kappa/2 +o(1)}   
 \end{equation}
 (assuming that $\kappa$ is small enough). 
Substituting~\eqref{eq:V1} and~\eqref{eq:V2} in~\eqref{eq:UUkU0}, leads us 
to the bound 
$$
U \ll HLT^2  p^{-\varepsilon}\log p +  
H   L  T^2 p^{ -\kappa +o(1)}
$$
and  the result follows.
\end{proof}

Let $E(\cG)$ be the additive energy of a 
 multiplicative subgroup   $\cG \subseteq \F_p^*$, that is
$$
E(\cG) = \# \{(\lambda_1,\mu_1,\lambda_2, \mu_2) \in\cG^4~:~
\lambda_1+ \mu_1 = \lambda_2 + \mu_2 \}.
$$
By a result of Heath-Brown and Konyagin~\cite{HBK}, 
if $\#\cG = T \le p^{2/3}$
we have 
$$
E(\cG) \ll T^{5/2}.
$$
Recently, Shkredov~\cite{Shkr1} has given an 
improvement which we present in the following 
slightly less precise form (which supreses logarithmic 
factors in $p^{o(1)}$). 

\begin{lem}
\label{lem:EG}  For  $T \le p^{2/3}$ we have
$$
E(\cG) \le 
\left\{ \begin{array}{ll}
T^{32/13} p^{o(1)},&  \textrm{if $T \le p^{1/2}$},\\
T^{31/13} p^{1/26+o(1)},&  \textrm{if $p^{1/2} < T \le p^{29/48}$},\\
T^{3} p^{-1/3+o(1)},&  \textrm{if $p^{29/48} < T \le p^{2/3}$}.\\
\end{array} \right.
$$
\end{lem}

. 

\section{Proofs of main results}

\subsection{Proof of Theorem~\ref{thm:Bound IG1}}

We have 
\begin{equation}
\label{eq:S W}
S_\chi(a, \cI, \cG) = \frac{1}{T} W,
\end{equation}
where
$$
W =  \sum_{x \in \cI} \sum_{\lambda, \mu  \in \cG} \overline\chi(\mu)
\chi(\mu x + a\lambda).
$$
(since  $\overline \chi(\mu) = \chi(\mu^{-1})$ for $\mu \in \F_p^*$).
Hence
$$
|W| \le   \sum_{x \in \cI} \sum_{\lambda,\mu  \in \cG} \left|
\sum_{\lambda   \in \cG}
\chi(x\mu + a\lambda)\right|.
$$
Collecting the products $\mu x$ with the same value $u \in \F_p$, 
we obtain 
$$
|W| \le \sum_{u \in \F_p} R(u) \left| \sum_{\lambda  \in \cG} 
\chi(u + a \lambda) \right|,
$$
where 
$$
R(u) = \# \{(x,\mu) \in \cI  \times \cG~:~
 \mu x= u  \}.
$$
So, by the Cauchy inequality, 
$$
|W|^2  \le  \sum_{u \in \F_p}R(u)^{2} 
\sum_{u \in \F_p} \left|\sum_{\lambda  \in \cG} 
\chi(u + a \lambda)\right|^{2} .
$$
Thus applying Lemma~\ref{lem:DavErd2} we derive 
$$
W^{2} \le  pT \sum_{u \in \F_p}R(u)^{2}. 
$$
Clearly
$$
\sum_{u \in \F_p}R(u)^{2} = Q , 
$$
where 
$$
Q = \# \{(x,y,\lambda,\mu) \in \cI \times \cI \times \cG \times \cG~:~
 \lambda x= \mu y  \}.
$$
Furthermore, it is clear that $Q = T N(\cI,\cG)$, where 
$N(\cI,\cG)$ is as in Lemma~\ref{lem:NIG}. 
Putting everything together and using the bound of 
 Lemma~\ref{lem:NIG}, we see that for any fixed $\nu\ge 1$ we have 
\begin{equation}
\label{eq:W2}
W^{2} \le  pT^2 \(  H t^{(2\nu +1)/2\nu(\nu+1)}p^{-1/2(\nu +1) + o(1)}
+ H^2 t^{1/\nu}p^{-1/\nu + o(1)}\),
\end{equation}
where 
$$
t = \max\{T, p^{1/2}\}.
$$
We can certainly assume that $T \le p^{1/2 + \varepsilon}$ as 
otherwise the result follows from the bound~\eqref{eq:CharSum}.
Thus $t \le p^{1/2 + \varepsilon}$ and we obtain 
$$
W^{2} \le  pT^2 \(  H p^{1/4\nu(\nu +1) + \varepsilon (2\nu +1)/2\nu(\nu+1) + o(1)}
+ H^2 p^{-1/2\nu + \varepsilon/\nu + o(1)}\).
$$
Since $H \ge p^\varepsilon$, taking a sufficiently large $\nu$ we can achieve the inequality
$$
H p^{1/4\nu(\nu +1) + \varepsilon (2\nu +1)/2\nu(\nu+1)}\le
H^2 p^{-1/2\nu + \varepsilon/\nu}.
$$
We can also assume that $\varepsilon< 1/3$ as otherwise the 
result follows from the Burgess bound, see~\cite[Theorem~12.6]{IwKow}, 
so the bound becomes 
$$
W^{2} \le  H^2 T^2 p^{1-1/6\nu + o(1)} \ll H^2 T^2 p^{1-1/7\nu} .
$$
Recalling~\eqref{eq:S W}, we obtain 
$$
S_\chi(a, \cI, \cG)  \ll H p^{1/2-1/7\nu} \le HT p^{-1/14\nu}
$$
for $T \ge p^{1/2-1/14\nu}$. 

\subsection{Proof of Theorem~\ref{thm:Bound IG2}}

Clearly we can assume that $H< p^{1/3}$ as otherwise the Burgess bound,
(see~\cite[Theorem~12.6]{IwKow}) implies the desired result. We can also assume
that $\varepsilon>0$ is small enough, 
thus the conditions of Lemma~\ref{lem:Uv} are satisfied.

We set 
$$
\gamma = \eta/3,
$$ 
where $\eta$ is as in Lemma~\ref{lem:Uv} (which we assume to e sufficiently 
small). 

Let  $L = Hp^{-2\gamma}$, $R = \rf{p^{\gamma}}$,and let $\cL$ be the set of primes of the interval $[L, 2L]$.

Clearly 
\begin{equation}
\label{eq:Sigma}
\bS_\chi(f, \cI, \cG) = \frac{1}{\#\cL R} 
\Sigma + O(LR T) = 
\frac{1}{\#\cL R} 
\Sigma + O(HTp^{-\gamma}),
 \end{equation}
where
\begin{equation*}
\begin{split}
\Sigma & = 
\sum_{\ell \in \cL} \sum_{r =1}^R
\sum_{x \in \cI} \sum_{\lambda \in \cG} 
\chi(x + f(\lambda) + \ell r) \\
& \le  \sum_{\ell \in \cL} 
\sum_{x \in \cI} \sum_{\lambda \in \cG} 
\left|\sum_{r =1}^R
\chi\(\frac{x + f(\lambda)}{\ell} + r\)\right|
=  \sum_{v \in \F_p} U(v)
\left|\sum_{r =1}^R
\chi\(v + r\)\right|, 
\end{split}
\end{equation*}
where $U(v)$ is given by~\eqref{eq:def Uv}. 
We now fix some integer $\nu \ge 1$ 
Writing $U(v) = U(v)^{(\nu-1)/\nu} (U(v)^2)^{1/2\nu}$
and using the H{\"o}lder inequality, we derive
\begin{equation*}
\begin{split}
\Sigma^{2\nu} & = \( \sum_{v \in \F_p} U(v)\)^{2\nu - 2}
\sum_{v \in \F_p} U(v)^2 
 \sum_{v \in \F_p}  
 \left|\sum_{r =1}^R
\chi\(v + r\)\right|^{2\nu}. 
\end{split}
\end{equation*}
We obviously have 
$$
\sum_{v \in \F_p} U(v) \le H \# \cL T \ll HLT.
$$
Hence, using Lemmas~\ref{lem:DavErd1} and~\ref{lem:Uv} we derive
$$
\Sigma^{2\nu} \ll (HLT)^{2\nu -2} HLT^2 \(R^{2\nu} p^{1/2} + R^\nu p\).
$$
Taking $\nu$ sufficiently large (depending on $\gamma$), 
we arrive to the inequality
\begin{equation}
\label{eq:Sigma nu}
\Sigma^{2\nu} \ll (HL)^{2\nu -1}T^{2\nu} R^{2\nu} p^{1/2-\eta}
=  (HLRT)^{2\nu}  (HL)^{-1}p^{1/2-\eta}.
 \end{equation}
So taking $\delta = \kappa/4$, we see that 
$$
 (HL)^{-1}p^{1/2-\eta} = H^{-2} p^{1/2-2\eta/3}\le p^{-\eta/6}.
$$
Hence we infer from~\eqref{eq:Sigma nu} that 
$\Sigma \ll  (HLRT) p^{-\eta/12\nu}$, which after substitution 
in~\eqref{eq:Sigma} concludes the proof. 

\subsection{Proof of Theorem~\ref{thm:Bound IG Expl}}

We proceed as before and use that $t,T = p^{1/2+o(1)}$, so~\eqref{eq:W2}
becomes 
$$
W^{2} \le  p^2 \(   p^{1/4 + 1/4\nu(\nu +1) + o(1)}
+  p^{1/2 -1/2\nu + o(1)}\). 
$$
Taking $\nu = 2$ we obtain 
$$
W^{2} \le  p^2 \(   p^{1/4 + 1/24+ o(1)}
+  p^{1/4 + o(1)}\) = p^{55/24+ o(1)},  
$$
which after substitution in~\eqref{eq:S W} implies the result. 

\subsection{Proof of Theorem~\ref{thm:Bound GG}}

As before, we have 
\begin{equation}
\label{eq:T W}
T_\chi(a,  \cG) = \frac{1}{T} W,
\end{equation}
where
$$
W = \sum_{\lambda,\mu, \vartheta \in \cG} \overline\chi(\vartheta)
\chi(a\vartheta + \mu + \lambda).
$$
Hence
$$
|W| \le   \sum_{\lambda,\mu  \in \cG} \left|
\sum_{\vartheta   \in \cG}\overline\chi(\vartheta)
\chi(a\vartheta  + \lambda + \mu)
\right|.
$$
Collecting the sum $\lambda+\mu$ with the same value $u \in \F_p$, 
we obtain 
$$
|W| \le \sum_{u \in \F_p} F(u) \left| \sum_{\lambda  \in \cG} 
\chi(a\vartheta + \lambda + \mu) \right|,
$$
where 
$$
F(u) = \# \{(\lambda,\mu) \in \cG^2~:~
\lambda + \mu = u  \}.
$$
So, as in the proof of Theorem~\ref{thm:Bound IG1} we obtain
$$
W^{2} \le  pT \sum_{u \in \F_p}R(u)^{2} = pT E(\cG). 
$$
Recalling Lemma~\ref{lem:EG} and using~\eqref{eq:T W}, we conclude the proof.

\subsection{Proof of Theorem~\ref{thm:up}}

Let us fix an arbitrary $\varepsilon > 0$.
Let $\ell_p$ denote the multiplicative order
of $2$ modulo $p$. We see from Theorem~\ref{thm:Bound GG} that if for a sufficiently
large prime $p$ we have $\ell_p \ge p^{13/33+\varepsilon}$ then 
$$
 \sum_{1 \le k < m \le \ell_p} 
\chi(2^m + 2^k +1) = 
 \sum_{k,m =1}^{\ell_p} 
\chi(2^m + 2^k +1) + O(\ell_p) 
= O(\ell_p^{2-\delta}). 
$$ 
Using a standard 
method of detecting primitive roots via multiplicative 
charactes,  we conclude that if  for a sufficiently
large prime $p$ we have $\ell_p \ge p^{13/33+\varepsilon}$
then $u_p \le 3$. It remains to estimate the number of 
primes $p \le Q$ with $\ell_p \ge p^{13/33+\varepsilon}$. 
Let $L = Q^{13/33+\varepsilon}$. Clearly for every such prime 
we have $p \mid W$ 
where 
$$
W = \prod_{\ell \le L} (2^\ell -1) \le 2^{L(L+1)/2}.
$$
Since $W$ has $O(\log W) = O(L^2) = O(Q^{26/33+2\varepsilon})$
prime divisors and since $\varepsilon$ is arbitrary, the result now follows. 
 
\section{Comments}

It is easy to see that the full analogues of Theorems~\ref{thm:Bound IG1}
and~\ref{thm:Bound IG2} can also be obtained for the sums
$$
\sum_{x \in \cI} \sum_{\lambda \in \cG} 
\chi\(\lambda x+ a\), \qquad  1\le a < p-1,
$$
without any changes in the proof. Using a version 
of Lemma~\ref{lem:NIG} given in~\cite[Lemma~9]{KonShp},
one can also obtain analogues of our results for sums over 
the consecutive powers $g, \ldots, g^N$ of a fixed element $g \in \F_p^*$,
provided that $N$ is smaller than the multiplicative 
order of $g$ modulo $p$ and in the same ranges as $T$ 
in Theorems~\ref{thm:Bound IG1} and~\ref{thm:Bound IG2}. 

Furthermore,  without any changes in the proof, 
Theorem~\ref{thm:Bound IG2} can extended to the double sums
$$
\sum_{x \in \cI} \sum_{u \in \cU} 
\chi\(ax+ u\), \qquad  1\le a < p-1,
$$
where $\cU \subseteq\F_p$ is an arbitrary set 
of cardinality
 $U \ge p^\varepsilon$ and an interval $\cI$ of length 
 $H \le p^{1/3}$, such that for some  $\kappa > 0$ 
 we have
$$
\# \{u \in \cU~:~ u \equiv b+x \pmod p, \text{ where } x \in \cI\}
\ll U^{1-\kappa}
$$
(which replaces Lemma~\ref{lem:G in I} in our argument).

It is also interesting to estimate sums 
\begin{equation}
\label{eq:GenSum}
\sum_{x \in \cI} \sum_{\lambda \in \cG} 
\chi\(f(x)+ \lambda \), \qquad  1\le a < p-1,
\end{equation}
with a nontrivial polynomial $f(X)\in \F_p[X]$, for  $H > p^{1/2-\eta}$ 
and $\# \cG  > p^{1/2-\eta}$ for some fixed $\eta>0$ (depending 
only on $\deg f$). To estimate these sums, one needs a nontrivial bound on the number of solutions to the congruence
$$
\lambda f(x)\equiv f(y) \pmod p, \qquad  x,y \in \cI, \  \lambda\in \cG, 
$$
which is better than $H^2$. In fact,  using some ideas 
and results of~\cite{GomShp,Shp} one can get such a bound, but not 
in a range in which the sums~\eqref{eq:GenSum} can be estimated nontrivially. 

Finally, it is interesting to investigate whether one can estimate the sums
$$
 \sum_{\lambda_1, \ldots, \lambda_\nu \in \cG} 
\chi(a + \lambda_1 +  \ldots + \lambda_\nu), \qquad  1\le a < p-1,
$$
with $\nu\ge 3$ in a shorter range than that of Theorem~\ref{thm:Bound GG}
by using bounds on the higher order additive energy of multiplicative
subgroups, see~\cite{Shkr1,Shkr2} for such bounds. Clearly, for any $\varepsilon> 0$
if $\#\cG > p^\varepsilon$ then for a sufficiently large $\nu$ such a result 
follows instantly from~\cite{BGK}, as if $\nu$ is large enough, the sums $\lambda_1 +  \ldots + \lambda_\nu$, $\lambda_1, \ldots, \lambda_\nu \in \cG$, represent each
element of $\F_p$ with the asymptotically equal frequency. We however hope that 
the approach via the higher order additive energy can lead to better estimates for smaller values of $\nu$ and $\varepsilon$. 

\section*{Acknowledgements}

During the preparation of this paper, the first author was supported by the 
NSF Grants~DMS~1301608 and by the NSF Grant~0932078000 while she was in residence at the Mathematical Science Research Institute in Berkeley, California, during the spring 2014 semester. 
This
 author would also like to thank the Mathematics Department of
the University of California at Berkeley for its hospitality.

The second author was supported  by the ARC 
Grant~DP130100237. 
This author would also to thank the Max Planck Institute for
Mathematics, Bonn, for support and hospitality during his work 
on this project.

\end{document}